\documentclass[11pt]{amsart}
\usepackage{enumerate}
\usepackage{comment}
\usepackage{graphicx}
\usepackage{amssymb}
\usepackage{xcolor}

\newtheorem{theorem}{Theorem}

\newtheorem{definition}[theorem]{Definition}

\newtheorem{corollary}[theorem]{Corollary}

\newtheorem{question}{Question}

\newcommand{\comm}[1]{}
\newcommand{\beqa}{\begin{eqnarray*}}
\newcommand{\eeqa}{\end{eqnarray*}}
\newcommand{\beqn}{\begin{eqnarray}}
\newcommand{\eeqn}{\end{eqnarray}}
\newcommand{\e}{\varepsilon}

\newcounter{cnt1}
\newcounter{cnt2}
\newcounter{cnt3}
\newcounter{cnt4}
\newcommand{\blr}{\begin{list}{$($\roman{cnt1}$)$}
{\usecounter{cnt1} \setlength{\topsep}{0pt}
\setlength{\itemsep}{0pt}}}
\newcommand{\bla}{\begin{list}{$($\alph{cnt2}$)$}
{\usecounter{cnt2} \setlength{\topsep}{0pt}
\setlength{\itemsep}{0pt}}}
\newcommand{\bln}{\begin{list}{$($\arabic{cnt3}$)$}
{\usecounter{cnt3} \setlength{\topsep}{0pt}
\setlength{\itemsep}{0pt}}}
\newcommand{\blR}{\begin{list}{$($\Roman{cnt4}$)$}
{\usecounter{cnt4} \setlength{\topsep}{0pt}
\setlength{\itemsep}{0pt}}}
\newcommand{\el}{\end{list}}

\begin{document}

\title[On Hereditary MIP]{A short note on hereditary Mazur intersection property}

\author[Bandyopadhyay]{Pradipta Bandyopadhyay}
\address[Pradipta Bandyopadhyay]{Stat--Math Division,
Indian Statistical Institute, 203, B.~T. Road, Kolkata
700108, India.}
\email{pradipta@isical.ac.in}

\author[Gothwal]{Deepak Gothwal}
\address[Deepak Gothwal]{Department of Mathematics, Indian Institue of Technology, Kharagpur, 
721302, India.}
\email{deepakgothwal190496@gmail.com}


\date{\today}


\begin{abstract}
In this note, we prove that hereditary Mazur intersection property (MIP) does not imply Fr\'echet smoothness using an example by Borwein and Fabian in \cite{BF}.
\end{abstract}

\maketitle

\section{Introduction}
\begin{definition} \rm
A Banach space $(X, \|\cdot\|)$ is said to have the \emph{Mazur intersection property (MIP)} if every closed, bounded, and convex set in $(X, \|\cdot\|)$ is the intersection of the closed balls containing it.
\end{definition}

MIP has been extensively studied over several decades and still continues to draw attention due to its rich geometric influence on the Banach space in question. A complete characterisation of MIP was obtained in \cite[Theorem 2.1]{GGS}. Earlier, Mazur himself \cite{Ma} proved that Fr\'echet smoothness of $(X, \|\cdot\|)$ implies MIP. The converse fails even in finite dimensions (see \cite[Page 43]{Ba}). But, a separable MIP space has separable dual \cite[Theorem 2.1 (iii)]{GGS}, and hence, is an Asplund space. Therefore, it has a Fr\'echet smooth renorming \cite[Theorem II.3.1]{DGZ}. Even this connection breaks down in non-separable Banach spaces. The fact that every Banach space can be isometrically embedded in a Banach space with MIP \cite{MS1} adds some notoriety to the behaviour of spaces with MIP. In particular, this shows that MIP is not hereditary.

On the other hand, Fr\'echet smoothness is clearly a hereditary property. Thus, Mazur's result actually shows that Fr\'echet smoothness of $(X, \|\cdot\|)$ implies that it has the hereditary MIP, that is, $X$ and all its subspaces have MIP.

Now, in this form, is the converse of Mazur's result true? That is, does the hereditary MIP imply Fr\'echet smoothness? Notice that if $(X, \|\cdot\|)$ has the hereditary MIP, then \bla
\item Every separable subspace of $(X, \|\cdot\|)$ has MIP, and hence, has separable dual. Therefore, $X$ is Asplund \cite[Theorem 2.14]{Ph2}.
\item $(X, \|\cdot\|)$ is smooth. To see this, note that for two dimensional spaces, smoothness and MIP are equivalent \cite{Ph}. Also, if every two dimensional subspace of $(X, \|\cdot\|)$ has MIP, and hence, is smooth, then $(X, \|\cdot\|)$ is smooth \cite[Propositon 5.4.21]{Me}. Hence, our question has an affirmative answer in finite dimensions.
\item The duality map on $(X, \|\cdot\|)$ and all its subspaces are ``quasi-continuous'' \cite[p 114]{GGS}.
\el
Thus, it seems natural to conjecture that hereditary MIP implies Fr\'echet smoothness.

In this note, we obtain some sufficient conditions for hereditary MIP in terms of the set $N(X)$ of non-Fr\'echet smooth points of $S(X)$. In particular, we show that if $(X, \|\cdot\|)$ is smooth and $N(X)$ is finite, then $(X, \|\cdot\|)$ has the hereditary MIP.

Borwein and Fabian in \cite[Theorem 4]{BF} proved an interesting result that an infinite dimensional WCG Asplund space $X$ admits an equivalent smooth norm $\|\cdot\|$ such that $N(X) = \{\pm x_0\}$ for some $x_0 \in S(X,\|\cdot\|)$. Thus, our results show that the Borwein-Fabian norm gives a counterexample to the above conjecture.

A preliminary version of this note is contained in the second author's Ph.~D.\ thesis \cite{GoTh} written under the supervision of the first author. We would like to thank Prof. Gilles Godefroy for drawing our attention to \cite{BF}.

Any undefined notions may be found in \cite{DGZ} or \cite{Me}.

\section{Main Results}
We consider real Banach spaces only. Let $(X,\|\cdot\|)$ be a Banach space and $X^*$ its dual. By a subspace of $(X,\|\cdot\|)$, we mean a closed linear subspace. For $x\in X$ and $r>0$, we denote by $B(x, r)$ \emph{the open ball} $\{y\in X : \|x-y\|<r\}$ and by $B[x, r]$ \emph{the closed ball} $\{y\in X: \|x-y\|\leq r\}$ in $(X,\|\cdot\|)$. We denote by $B(X)$ the \emph{closed unit ball} $\{x\in X : \|x\| \leq 1\}$ and by $S(X)$ the \emph{unit sphere} $\{x \in X : \|x\| = 1\}$.
For $x\in S(X)$, let
\[
D_X (x) = \{f\in S(X^*) : f(x)=1\}.
\]
The multivalued map $D_X$ is called the \emph{duality map}. Any selection of $D_X$ is called a \emph{support mapping}.

We say that $x \in S(X)$ is a \emph{smooth point} if $D_X(x)$ is a singleton. A Banach space $(X, \|\cdot\|)$ is \emph{smooth} if every $x \in S(X)$ is a smooth point.

We say that $x \in S(X)$ is a \emph{Fr\'echet smooth point} if the duality map $D_X$ is single-valued and norm-norm continuous at $x$. A Banach space $(X, \|\cdot\|)$ is \emph{Fr\'echet smooth} if every $x \in S(X)$ is a Fr\'echet smooth point. Let
\beqa
F(X) & := & \{x \in S(X) : \text{the norm is Fr\'echet smooth at } x\} \text{ and}\\
N(X) & := & \{x \in S(X) : \text{the norm is not Fr\'echet smooth at } x\}.
\eeqa

The following necessary or sufficient conditions for MIP follows from \cite[Theorem 2.1]{GGS} and are essentially already observed in \cite{Ph}.

\begin{theorem} \label{char}
For a Banach space $(X, \|\cdot\|)$, consider the following statements~:
\bla
\item $D_X(F(X))$ is norm dense in $S(X^*)$.
\item $(X, \|\cdot\|)$ has MIP.
\item The set of extreme points of $B(X^*)$ is norm dense in $S(X^*)$.
\el
Then $(a) \implies (b) \implies (c)$.
\end{theorem}

A Banach space $X$ is called \emph{Weakly Compactly Generated (WCG)} if there is a weakly compact subset $K$ of $X$ such that $\overline{span}(K)=X$.

In particular, all reflexive spaces and all separable spaces are WCG.

Borwein and Fabian in \cite[Theorem 4]{BF} proved that every infinite dimensional WCG Asplund space $X$ admits an equivalent smooth norm $\|\cdot\|$ such that $N(X) = \{\pm x_0\}$ for some $x_0 \in S(X, \|\cdot\|)$.

\begin{theorem} \label{subspacecharthm}
Let $(X, \|\cdot\|)$ be a Banach space such that $N(X) = \{\pm x_0\}$ for some $x_0 \in S(X)$. Then a subspace $Y \subseteq X$ with dim $Y \geq 2$ has MIP if and only if $x_0 \notin Y$ or $D_Y(x_0)$ has empty interior (relative to $S(Y^*)$).
\end{theorem}

\begin{proof}
Assume that $x_0 \notin Y$. Then, $Y$ is Fr\'echet smooth, and hence, has MIP \cite{Ma}.

Now, assume that $x_0 \in Y$ and $D_Y(x_0)$ has empty interior. Note that $N(Y) \subseteq N(X) \cap Y$.  We will verify Theorem~\ref{char} $(a)$.

Let $\e >0$ and $f \in S(Y^*)$ be given. Now, $(B(f, \e) \cap S(Y^*)) \setminus D_Y(N(Y))$ is a nonempty open set in $S(Y^*)$. By Bishop-Phelps Theorem, there exists $y_0 \in S(Y) \setminus N(Y)$ such that $\|f-f_{y_0}\| <\e$, where $f_{y_0} \in D_Y(y_0)$. Now, $y_0 \notin N(Y)$. So, $y_0 \in F(Y)$. Hence, by Theorem~\ref{char} $(a) \implies (b)$, $Y$ has MIP.

Conversely, let $Y$ has MIP and $x_0 \in Y$. Suppose $D_Y(x_0)$ has a non-empty interior in $S(Y^*)$. Since $D_Y(x_0)$ is a face of $B(Y^*)$, none of the interior points of $D_Y(x_0)$ can be an extreme point of $B(Y^*)$. Thus, Theorem~\ref{char} $(c)$
does not hold for $Y$. Hence, $Y$ cannot have MIP.
\end{proof}

\begin{corollary} \label{thmheremip}
Let $(X, \|\cdot\|)$ be such that $N(X) = \{\pm x_0\}$ for some $x_0 \in S(X)$. Then the following are equivalent: \bla
\item $(X, \|\cdot\|)$ has the hereditary MIP.
\item $D_X(x_0)$ is a singleton.
\item For every subspace $Y$ of $(X, \|\cdot\|)$ such that dim $Y \geq 2$ and $x_0 \in Y$, $D_Y(x_0)$ has empty interior (relative to $S(Y^*)$).
\el
\end{corollary}

\begin{proof}
$(a) \implies (b)$. We have already noted that $(X, \|\cdot\|)$ has the hereditary MIP implies $(X, \|\cdot\|)$ is smooth, and hence, $D_X(x_0)$ is a singleton.

$(b) \implies (c)$. Let $Y$ be a subspace of $(X, \|\cdot\|)$ such that dim $Y \geq 2$ and $x_0 \in Y$. Since $D_X(x_0)$ is singleton, it follows that $D_Y(x_0)$ is also singleton. Hence, $D_Y(x_0)$ has empty interior.

$(c) \implies (a)$ follows from Theorem~\ref{subspacecharthm}.
\end{proof}

\begin{corollary}
The Borwein-Fabian norm on a WCG Asplund space has hereditary MIP, but is not Fr\'echet smooth.

In particular, every Banach space with a separable dual, e.g., $c_0$, has an equivalent norm with the hereditary MIP, that is not Fr\'echet smooth.
\end{corollary}

Can $N(X)$ be larger, retaining hereditary MIP? The following result gives a sufficient condition.

\begin{theorem} \label{t1}
If $(X, \|\cdot\|)$ is smooth and $N(X)$ is finite, then $(X, \|\cdot\|)$ has hereditary MIP. In particular, $(X, \|\cdot\|)$ is Asplund.
\end{theorem}

\begin{proof}
Let $Y$ be a subspace of $(X, \|\cdot\|)$. If dim $Y <\infty$ or $N(Y) = \emptyset$, then $Y$ is Fr\'echet smooth; hence has MIP.

Now assume that $Y$ is infinite dimensional and $N(Y) \neq \emptyset$.

Let $\e>0$ and $f \in S(Y^*)$. Since $D_Y$ is single-valued, $D_Y(N(Y))$ is finite. So, $A := (B_Y(f, \e/2) \cap S(Y^*)) \setminus D_Y(N(Y)) \neq \emptyset$. Let $g \in A$. Then, $\text{dist}(g, D_Y(N(Y))) >0$. Choose $0 <\alpha < \min\{\e/2, \text{dist}(g, D_Y(N(Y)))\}$. By Bishop-Phelps theorem, there exist $y \in S(Y)$ and $h \in D_Y(y)$ such that $\|h-g\| <\alpha$. Clearly, $h \notin D_Y(N(Y))$. So, $y \notin N(Y)$ and hence, $y \in F(Y)$. Also, $\|f-h\| <\e$. Therefore, by Theorem~\ref{char} $(a) \implies (b)$, $Y$ has MIP.
\end{proof}

\begin{question}
Are there examples of Banach spaces satisfying the hypothesis of Theorem~\ref{t1}, with $N(X)$ having more than two points?
\end{question}

\begin{question}
If $(X, \|\cdot\|)$ has hereditary MIP, can $N(X)$ be infinite?
\end{question}

\begin{question}
If $(X, \|\cdot\|)$ is smooth and $N(X)$ is compact, can we conclude that $(X, \|\cdot\|)$ has hereditary MIP or is at least Asplund?
\end{question}

\end{document}